\documentclass[12pt]{amsart}
\usepackage{Baskervaldx}
\usepackage[cal=cm]{mathalfa}
\usepackage{enumerate, geometry, pgfplots, graphicx, subcaption, amsmath, array, pdfsync, tikz, tikz-cd, color, url, float, wrapfig, comment}
\usepackage[baskervaldx,cmintegrals,bigdelims,vvarbb]{newtxmath}

\usepackage[hyperfootnotes=false, pdfencoding=auto]{hyperref}
\hypersetup{
  colorlinks= true,%Colors links instead of ugly boxes
  linkcolor=[RGB]{0,89,42},
  urlcolor=[RGB]{0,89,42},
  citecolor=[RGB]{0,89,42}
}

\usepackage[font=footnotesize]{caption}

\newcommand\david[1]{{\color{orange} \sf $\Delta$ David: [#1]}}

\usetikzlibrary{calc}

\numberwithin{equation}{section}

\title{Rational endomorphisms of Fano hypersurfaces}
\author{Nathan Chen and David Stapleton}

\definecolor{mycolor}{RGB}{146, 214, 203}
\definecolor{myothercolor}{RGB}{179, 215, 232}

\newtheorem{theorem}{Theorem}
\newtheorem{corollary}[theorem]{Corollary}
\newtheorem{lemma}[theorem]{Lemma}
\newtheorem{proposition}[theorem]{Proposition}

\numberwithin{theorem}{section}

\newcommand\blfootnote[1]{%
  \begingroup
  \renewcommand\thefootnote{}\footnote{#1}%
  \addtocounter{footnote}{-1}%
  \endgroup
}

\newtheorem{Lthm}{Theorem}

\newtheorem{Lcor}[Lthm]{Corollary}

\theoremstyle{definition}

\newcommand{\hr}[2]{\hyperref[#1]{#2}}

\theoremstyle{definition}
\newtheorem{remark}[theorem]{Remark}

\newtheorem{example}[theorem]{Example}
\newtheorem{definition}[theorem]{Definition}

\def\ZZ{{\mathbb Z}}
\def\QQ{{\mathbb Q}}

\def\AA{{\mathbb A}}

\def\CC{{\mathbb C}}
\def\LL{{\mathbb L}}
\def\TT{{\mathbb{T}}}

\def\PP{{\mathbb{P}}}

\def\Pic{{\mathrm{Pic}}}
\def\Cl{{\mathrm{Cl}}}
\def\Spec{{\mathrm{Spec}}}

\def\Oc{{\mathcal{O}}}
\def\Mc{{\mathcal{M}}}
\def\Lc{{\mathcal{L}}}

\def\Nc{{\mathcal{N}}}

\def\Gm{{\mathbb{G}_m}}
\def\zhens{{{\ZZ}_p^{sh}}}
\def\zhenstwo{{\ZZ_2^{sh}}}
\def\qhens{{\QQ_p^{sh}}}

\def\ord{{\mathrm{ord}}}
\def\wb{{\mathrm{WB}}}

\def\wbhat{{\widehat{\wb}}}
\def\nuhat{{\hat{\nu}}}
\def\xbar{{\overline{x}}}
\def\xs{{x_1,\ldots,x_n}}
\def\Fbar{{\overline{\mathbb{F}}_p}}
\def\Fbartwo{{\overline{\mathbb{F}}_2}}

\def\mf{{\mathfrak{m}}}
\def\RatEnd{{\mathrm{RatEnd}}}
\def\Bir{{\mathrm{Bir}}}

\def\dra{{\dashrightarrow}}
\def\ra{{\rightarrow}}
\def\etabar{{\overline{\eta}}}
\def\kapbar{{\overline{\kappa}}}
\def\gammabar{{\overline{\Gamma}}}
\def\Xbar{{\overline{X}}}
\def\Tbar{{\overline{\TT}}}
\def\cl{{\colon}}

\def\Frac{{\mathrm{Frac}}}

\pgfplotsset{compat=1.15}

\newcommand{\mybigwedge}{\raisebox{.3ex}{\scalebox{0.75}{$\bigwedge$}}}

\begin{document}

\maketitle

\thispagestyle{empty}

\begin{abstract}
We show that the degrees of rational endomorphisms of very general complex Fano and Calabi-Yau hypersurfaces satisfy certain congruence conditions by specializing to characteristic p. As a corollary we show that very general $n-$dimensional hypersurfaces of degree $d\ge \lceil5(n+3)/6\rceil$ are not birational to elliptic fibrations. A key part of the argument is to resolve singularities of general p-cyclic covers in mixed characteristic p.
\end{abstract}

The goal of this paper is to use specialization to characteristic $p$ techniques to give restrictions on the degrees of rational endomorphisms of complex Fano hypersurfaces. \blfootnote{The first author's research is partially supported by the National Science Foundation under the Stony Brook/Simons Center for Geometry and Physics RTG grant DMS-1547145.\\
\vspace{-12pt}

The second author's research is partially supported by the NSF FRG grant number 1952399.}

Koll\'ar \cite{Kollar96} showed that Fano varieties that are $p$-fold cyclic covers in characteristic $p$ can carry holomorphic forms. The positivity of these forms is a powerful tool for studying these varieties, and by specializing to characteristic $p$ one can prove results about Fano varieties in characteristic 0. Let $X\subset \PP^{n+1}_\CC$ be a very general hypersurface of degree $d$. Koll\'ar used specialization mod $p$ to prove that if $d\ge 2\lceil (n+3)/3 \rceil$ then $X$ is not rational (or even ruled). If $d\ge 2\lceil (n+2)/3 \rceil$ then Totaro \cite{Totaro} used specialization mod $p$ to show that $X$ is not even stably rational. In \cite{CS20}, the authors used specialization mod $p$ to prove that the degree of irrationality for Fano hypersurfaces can be arbitrarily large.

Let $\RatEnd(X)$ denote the set of rational endomorphisms of $X$ with degree at least 1, and let $\Bir(X)\subset \RatEnd(X)$ denote the subset of birational automorphisms. If $X$ is of general type then it follows from \cite{KobOch75} that $\Bir(X)=\RatEnd(X)$ (in other words, there are no rational endomorphisms of degree $\ge$ 2). Thus, studying degree $\ge 2$ endomorphisms on smooth hypersurfaces is only interesting in the Fano or Calabi-Yau range. On the other hand, every rational variety admits rational endomorphisms of every degrees, so one can view our main result as strengthening the nonrationality of the varieties involved.

\begin{Lthm}\label{mainThm}
Let $X \subset \PP^{n+1}$ be a very general hypersurface over $\CC$ of degree $d$ and dimension $n\ge 3$. Let $p$ be a prime number such that
\[ d\ge p \left\lceil \frac{n+3}{p+1} \right\rceil. \]
If $\phi$ is any rational endomorphism of $X$ of degree $\lambda$ then $\lambda \equiv 0 \text{ or } 1 \pmod{p}.$
\end{Lthm}

\noindent Very little is known about rational endomorphisms of nonrational varieties. Dedieu (\cite{Dedieu09}) has given constraints on degrees of rational endomorphisms of K3 surfaces. Our results are the first constraints on rational endomorphisms of nonrational Fano hypersurfaces.

As above, let $X\subset\PP_\CC^{n+1}$ be very general of degree $d$ and dimension $n$ with a rational endomorphism of degree $\lambda$. In low dimensions, we have the following constraints on $\lambda$:
\begin{center}
\begin{tabular}{c | c | c | c | c} 
 ($n,d$) & (3,5) & (4,6) & (5,6) & (5,7) \\
 \hline
 type & Calabi-Yau & Calabi-Yau & Fano & Calabi-Yau \\
 \hline
 $\lambda \equiv 0 \text{ or }1$ & $\pmod 5$ & $\pmod 3$ & $\pmod 3$  & $\pmod 3$ and$\pmod 7$ \\ 
\end{tabular}
\end{center}

Unirational varieties are sources of varieties with many rational endomorphisms. If $X$ is unirational, precomposing the unirational parametrization with any map $X\dra \PP^n$ gives
a rational endomorphism:
\[
X\dra \PP^n\dra X.
\]
For example, any smooth cubic hypersurface of dimension $\ge 2$ is unirational, and therefore admits many rational endomorphisms. Moreover, Beheshti and Riedl \cite{BR19} proved that if $X\subset\PP^{n+1}_\CC$ is smooth and $n\ge 2^{d!}$ then $X$ is unirational. So for any $d>0$ there exist degree smooth $d$ hypersurfaces of large dimension with many rational endomorphisms.  On the other hand, any rational endomorphism constructed this way will have degree divisible by the degree of the unirational parametrization. Koll\'ar showed (by specializing mod $p$) that if $X\subset \PP^{n+1}_\CC$ is very general of degree $d\ge p\lceil (n+3)/(p+1) \rceil$ then any unirational parametrization (in fact, any parametrization by a uniruled variety) has degree divisible by $p$. Indeed, it is natural to conjecture there are Fano hypersurfaces with no rational endomorphisms of degree $\ge2$. Such a variety would provide an answer to a famous open problem in that it would be an example of a rationally connected variety which is not unirational.

Another source of varieties with rational endomorphisms are varieties that birationally admit \textit{elliptic fibrations}, i.e. a variety $X$ with a rational map:
\[
X\dra B
\]
such that $X/\CC(B)$ is birational to a $\CC(B)$-elliptic curve $E_{\CC(B)}$ (with a $\CC(B)$-point). For example, any cubic hypersurface is birational to an elliptic fibration. Note that for any integer $m$, an elliptic curve $E_{\CC(B)}$ has a multiplication by $n$ map:
\[
m\cl E_{\CC(B)}\ra E_{\CC(B)},
\]
which has degree $n^{2}$. If $p$ is a prime and $n^2\equiv 0$ or 1$\pmod{p}$ for all $n \geq 2$, then $p$ is equal to 2 or 3. Thus the degree restrictions in Theorem~\ref{mainThm} are enough to prove the following:

\begin{Lcor}
Let $X\subset \PP^{n+1}_\CC$ be very general of degree $d$. If $d\ge 5 \lceil (n+3)/6 \rceil$ then $X$ is not birational to an elliptic fibration.
\end{Lcor}
\noindent For example, we see that a very general Calabi-Yau quintic threefold is not birational to an elliptic fibration. For Calabi-Yau hypersurfaces this was already known by work of Grassi and Wen (see \cite[Thm. 1.8]{Grassi91} and \cite{GrassiWen19}). Similarly, if $X$ is birational to a genus 1 fibration without a section but with a line bundle of relative degree $\delta$ then $X$ admits rational endomorphisms of degree $(m\delta+1)^2$ for all $m\ge 0$. Likewise, if $X$ is birational to a fibration by $g$-dimensional abelian varieties then $X$ admits rational endomorphisms of degree $m^{2g}$ for all $m\ge 0$. Theorem~\ref{mainThm} can be applied to give results in these settings as well.

The idea behind this paper is to spread out rational endomorphisms in characteristic $p$. The main technical difficulty that we encounter is proving that certain cyclic covers in mixed characteristic have ``mild singularities" after allowing for the possibility of an arbitrary finite base change. The technical result that we prove is:

\begin{Lthm}\label{thm:SusRuledMod}
Let $\zhens$ be a strict Henselization of the $p$-adic integers with field of fractions $\qhens$ and residue field $\Fbar$. Let $X$ be a smooth scheme over $\zhens$. Let $D \subset X$ be a divisor which is smooth over $\zhens$ such that $\Oc(D)\cong L^p$ for some line bundle $L$ on $X$. Let $Y$ be the $p$-cyclic cover of $X$ branched at $D$. Suppose the natural section $s\in H^0(X_{\Fbar},\Oc_{X_\Fbar}(D))$ has nondegenerate critical points \cite[17.3]{Kollar95}. Then $Y$ has sustained separably uniruled modifications (see Definition~\ref{ruledMod}).
\end{Lthm}

\noindent Roughly speaking $Y/\zhens$ has sustained separably uniruled modifications if for every finite base change $\zhens\subset R$, every exceptional divisor over $X_R$ is separably uniruled. The proof of Theorem~\ref{thm:SusRuledMod} involves resolving these mixed characteristic cyclic covers by a sequence of weighted blowups.

\noindent\textbf{Sketch of proof of Theorem A.} Consider a scheme $Y$ over $\zhens$ with mild singularities. We may spread out a rational endomorphism on the generic fiber by considering the (normalization of the closure of the) graph $\Gamma$ in $Y \times_{\zhens} Y$, with projections $\pi_{1}, \pi_{2}$. The central fiber of $\Gamma$ possibly breaks up into a union of components
\[ \Gamma_{0} \cup E_{1} \cup \cdots \cup E_{n}, \]
where $\Gamma_{0}$ is the unique component birational to $Y_{0}$ under $\pi_{1}$ and the $E_i$ are exceptional divisors over $Y$. The second projection then gives a rational endomorphism
\begin{center}
\begin{tikzcd}
\phi_{0} \colon Y_{0} \simeq_{\text{bir.}} \Gamma_{0} \arrow[r, dashed, "\pi_{2}"] & Y_{0}
\end{tikzcd}
\end{center}
of the central fiber. The assumption on mild singularities implies that all exceptional divisors $E_{i}$ are separably uniruled, but $Y_{0}$ is not separably uniruled. Therefore, we see that each map $\pi_{2} \big|_{E_{i}} \colon E_{i} \dashrightarrow Y_{0}$ must have degree divisible by $p$. Comparing the degree of the original rational endomorphism on the generic fiber to the degree of $\phi_{0}$ we see they are congruent mod $p$. Finally, we use the positivity properties of forms in characteristic $p$ to show that $Y_{0}$ has no separable rational endomorphisms of degree $\geq 2$.

\noindent{\textbf{Acknowledgments.}} We would like to thank Claire Voisin for raising this question. We thank Dan Abramovich for explaining weighted blowups in mixed characteristic. We would also like to thank Kenneth Ascher, Nguyen-Bac Dang, Kristin DeVleming, Fran\c{c}ois Greer, Kiran Kedlaya, J\'{a}nos Koll\'{a}r, Robert Lazarsfeld, John Lesieutre, Daniel Litt, James M\textsuperscript{c}Kernan, John Ottem, Ari Shnidman, and Burt Totaro for helpful conversations.

%%%% Add result on degrees of unirational parametrizations.

\begin{comment}
\noindent \textbf{Conventions}. By the graph of a rational map $\phi\cl X\dra X$ we mean the closure of the graph $\Gamma$, which is thought of as a subvariety in $X\times X$. We will frequently use $\pi_{1}$ and $\pi_{2}$ to denote the projection maps onto the first and second factors, respectively. By the degree of a (rational) map $f \cl X \dashrightarrow Y$, we mean the number of pre-images (counted with multiplicity) of a general point in $Y$.
\end{comment}

\section{Specializing rational endomorphisms}

The purpose of this section is to explain the relationship between uniruling degrees and specializations of endomorphisms for schemes with mild singularities. Throughout, let $R$ be a DVR, let $T=\Spec(R)$ with closed point $0\in T$, let $\eta=\Frac(R)$ be the field of fractions, and let $\kappa$ be the residue field. Let $\etabar$ and $\kapbar$ be their algebraic closures.

Many of the concepts below are discussed in \cite[IV.1.1 and VI.1.6]{Kollar96}.

\begin{definition}\label{ruledMod}
Let $X$ be a normal scheme. We say that $X$ has \textit{ruled modifications} if every exceptional divisor of every normal birational modification $f\cl Y\ra X$ is ruled. Likewise, we say that $X$ has \textit{separably uniruled modifications} if every exceptional divisor of every normal birational modification $f\cl Y\ra X$ is separably uniruled.
\end{definition}

\begin{example}\label{ex:regSchRulMod}
A regular scheme $X$ admits ruled modifications (see \cite{Abhyankar56}, or \cite[Thm. VI.1.2]{Kollar96}).
\end{example}

\begin{remark}
Having ruled modifications can be checked locally. To show a scheme $X$ admits (separably uni-)ruled modifications it suffices to find a proper birational map from a regular scheme:
\[
\nu\cl \Xbar\ra X
\]
such that all exceptional divisors are (separably uni-)ruled \cite[Thm. VI.1.7]{Kollar96}.
\end{remark}

We believe the following is well known to experts.

\begin{lemma}\label{etaleToric}
If $X$ is a complex variety that admits an \'etale map to a toric variety:
\[
\psi\cl X\ra \TT
\]
then $X$ admits ruled modifications.
\end{lemma}

\begin{proof}
Take a toric resolution of $\TT$; by this, we mean a morphism of toric varieties:
\[
\mu \cl \Tbar\ra \TT
\]
such that $\Tbar$ is smooth, $\mu$ is proper, and the morphism is defined by a refinement of the fan of $\TT$. Then (1) every exceptional divisor $E$ of $\mu$ is a toric variety, and (2) the map
\[
\mu|_E \cl E \ra B:=\mu(E)
\]
is a toric map. If $\CC(B)$ is the fraction field of $B$ then $E_{\CC(B)}$ is a rational $\CC(B)$-variety.

Now consider the following fiber square:
\begin{center}
\begin{tikzcd}
\Xbar\arrow[r]\arrow[d,"\nu"]&\Tbar\arrow[d,"\mu"]\\
X\arrow[r]&\TT
\end{tikzcd}
\end{center}
Then $\Xbar$ is smooth as it is \'etale over $\Tbar$. Let $D\subset \Xbar$ be an irreducible exceptional divisor of $\nu$ with image $E\subset \Tbar$ an exceptional divisor of $\mu$ and let $B'\subset X$ be the image of $D$ and $B\subset \TT$ the image of $E$. Then we have $D=E\times_B B'$ and as a result $D_{\CC(B')} = E_{\CC(B')}$ is rational over $\CC(B')$. So $D$ is ruled.
\end{proof}

For applications it will be necessary for our families $X_R$ to have ruled modifications after arbitrary finite base change $R\subset R'$. This will allow us to specialize properties from the geometric generic fiber $X_\etabar$ to the special fiber.

\begin{definition}
Let $X_R$ be a normal $R$-scheme. We say $X_R$ has \textit{sustained ruled modifications} if for every extension of DVRs $R\subset R'$ such that $\dim_\eta(\Frac(R'))$ is finite, there is a further extension $R'\subset R''$ with $\dim_\eta(\Frac(R''))$ finite such that $X_{R''}$ is a normal scheme with ruled modifications. We define \textit{sustained separably uniruled modifications} with appropriate substitutions.
\end{definition}

\begin{example}
If $X_R$ is a smooth $R$-scheme then $X_R$ is regular and for every extension $R\subset R'$ of DVRs we have $X_{R'}$ is a smooth $R'$-scheme, and is therefore regular. Thus by Example~\ref{ex:regSchRulMod}, $X_R$ admits sustained ruled modifications.
\end{example}

\begin{example}
Let $X_C \rightarrow C$ be a family of complex varieties over a smooth complex curve such that the total space $X_C$ is smooth. Let $R$ be the local ring at any point $0\in C$. Assume the central fiber $X_0\subset X_R$ is reduced with simple normal crossing. Then $X_R$ has sustained ruled modifications. Indeed let $R\subset R'$ be an extension of DVRs with uniformizers $t$ and $t'$ respectively. We can assume that $t=u(t')^k$ for some $k$ and unit $u\in R'$. After possibly extending $R'\subset R''$ to include the $k$-th root of $u$ we can assume that $t=(t'')^k$. By the simple normal crossing assumption, every point in $(X_{R''})_0$ has a neighborhood of the form
\[
U:=\left(t''^k=x_1\cdots x_m\right)\subset X\times \AA^1
\]
for some subsequence $x_1,\dots,x_m$ of a regular sequence on $X_R$. This gives (after possibly shrinking further) an \'etale map:
\[
U\ra \TT=\left(x_1\cdots x_m =(t'')^k\right)\subset \AA^{n+1}
\]
to a toric variety. By Lemma~\ref{etaleToric}, $X$ has sustained ruled modifications.
\end{example}

\begin{definition}
Let $X$ be a variety. A \textit{uniruling degree} of $X$ is the degree of a dominant and generically finite rational map $V\times \PP^1\dra X$.
\end{definition}

\begin{theorem}\label{SpecializeRatEnd}
Let $X_R$ be a normal $R$-scheme, assume $\kappa=\kapbar$ is algebraically closed, and assume $\eta$ is perfect.
\begin{enumerate}
    \item Suppose that $X_\kappa$ is reduced and irreducible. If $\ell$ divides every uniruling degree of $X_\kappa$ and $X_R$ has ruled modifications then for any rational endomorphism $\phi_\eta$ of the generic fiber $X_\eta$, there is an endomorphism $\phi_\kappa$ of $X_\kappa$ such that
    \[
    \deg(\phi_\eta)\equiv \deg(\phi_\kappa) \pmod{\ell}.
    \]
    \item Let us now assume that
    \[
    X_{\kappa} = D_{1} \cup \cdots\cup D_{m} \cup X'_{\kappa}
    \]
    is a reduced divisor and all the $D_i$ are ruled divisors. If $\ell$ divides every uniruling degree of $X'_\kappa$ and $X_R$ admits sustained ruled modifications, then for every rational endomorphism $\phi_\etabar$ of $X_\etabar$ there is a rational endomorphism $\phi_\kappa$ of $X'_\kappa$ such that
    \[
    \deg(\phi_\etabar)\equiv \deg(\phi_\kappa) \pmod{\ell}.
    \]
    \item Suppose that $\text{char } \kappa = p$ and $X_{\kappa}$ is reduced, irreducible, and not separably uniruled. If $X_R$ admits sustained separably uniruled modifications, then for every rational endomorphism $\phi_{\etabar}$ of $X_{\etabar}$, there is a rational endomorphism $\phi_{\kappa}$ of $X_{\kappa}$ such that
    \[ \deg(\phi_{\kappa}) \equiv \deg(\phi_{\kappa}) \pmod{p}. \]
\end{enumerate}
\end{theorem}

\begin{comment}
\david{in practice I think we only check that locally X has sustained ruled modifications. -- This assumption seems to be enough to prove the theorem, but it is not obvious to me that admitting locally sustained ruled modifications is the same as admitting sustained ruled modifications}
\end{comment}

\begin{proof}
For (1), consider the normalization of the closure of the graph of $\phi_\eta$,
\[
\gammabar_R\ra X_R\times_R X_R.
\]
Then $\gammabar_R$ is flat over $T$ and the central fiber $\Gamma_\kappa$ has 2 projections $\pi_1$ and $\pi_2$. By flatness, the projection map $\pi_1$ (resp. $\pi_2$) has degree 1 (resp. $\deg(\phi_\eta)$). Thus $\pi_1$ maps a unique component of $\Gamma_\kappa$ birationally to $X_\kappa$. We will show that $\phi_\kappa :=\pi_2\circ (\pi_1|_{X'}^{-1})$ has the required degree.

As $X_R$ has ruled modifications, every other component of $\Gamma_\kappa$ is ruled. Thus we may write
\[
\Gamma_\kappa = X'_\kappa\cup E_1 \cup \cdots \cup E_r
\]
where $X'_\kappa$ is the component mapping birationally to $X_\kappa$ and the reduced exceptional divisors $E_i^{\mathrm{red}}$ are all ruled. The degrees of the maps
\[
E_i\ra X_\kappa
\]
under the second projection are all either 0 or uniruling degrees. Therefore,
\begin{align*}
\deg(\phi_\eta) &= \deg(\pi_2) \\ &=\deg(\pi_2|_{X'_\kappa})+\sum_i \deg(\pi_2|_{E_i}) \\
&\equiv \deg(\pi_2|_{X'_\kappa}) \pmod{\ell}
\end{align*}

To prove (2), let $\eta \subset \eta'$ be a finite extension such that $\phi_\etabar$ is defined over $\eta'$. Moreover, assume there is a DVR $R'\subset \eta'$ such that $X_{R'}$ has ruled modifications and $\Frac(R')=\eta'$. Note that the residue field of $B$ is also $\kappa$ (as $\kappa$ is assumed to be algebraically closed). Consider the normalization of the closure of the graph of $\phi_{\eta'}$:
\[
\gammabar \ra X_{R'}\times_{R'}X_{R'}.
\]
The case when $\deg(\phi_\etabar)\equiv 0 \pmod{\ell}$ is uninteresting (by considering any constant map $X'_\kappa\ra X'_\kappa$). So we may assume that $\deg(\phi_\etabar)\not\equiv 0\pmod{\ell}$. As in the proof of (1), we have by flatness that the map $\pi_1$ (resp. $\pi_2$) has degree 1 (resp. $\deg(\phi_\etabar)$). So we may write
\[
\Gamma_\kappa = \Xbar'_\kappa\cup D'_1\cup \cdots D'_m \cup E_1 \cup \cdots \cup E_r
\]
where $\Xbar'_\kappa$ (resp. $D'_i$) is birational to $\Xbar_\kappa$ (resp. $D_i$) under $\pi_1$, and the $E_i$ are contracted by $\pi_1$. Then all the $E_i$ are ruled. The assumption that $\deg(\phi_\etabar)\not\equiv 0\pmod{\ell}$ and that $\ell$ divides any uniruling degree of $X_\kappa$ implies that $\Xbar'_\kappa$ maps dominantly and generically finitely onto $\Xbar_\kappa.$ Thus the composition
\[
\phi_\kappa: \Xbar_\kappa\simeq_{\mathrm{bir}} \Xbar'_\kappa\ra \Xbar_\kappa
\]
is a rational endomorphism of $\Xbar'_\kappa$ and a similar computation to part (1) proves
\[
\deg(\phi_\kappa) \equiv \deg(\phi_\etabar)\pmod{\ell}.
\]

The proof for (3) follows the same strategy as (1). The added hypothesis that $X_{\kappa}$ is not separably uniruled implies that every component $E_i$ of $\Gamma_{\kappa}$ other than the unique one birational to $X_{\kappa}$ (via $\pi_{1}$) is separably uniruled. Therefore each map $E_i\ra X_R$ has degree divisible by $p$. The proof then follows part (1).
\end{proof}

\begin{corollary}
Working over $\CC$, let $X \rightarrow B$ be a smooth family over a complex variety $B$. Fix integers $\ell \geq 2$ and $\lambda \geq 1$. Then the locus
\[ B_{(\ell, \lambda)} := \left\{ b \in B \ \middle| \ \let\scriptstyle\textstyle \substack{ \ell \text{ divides every uniruling degree of } X_{b} \text{ and } \not\exists \text{ a rational}\\ \text{ endomorphism } \varphi_{b}: X_{b} \dashrightarrow X_{b} \text{ with degree } \equiv\lambda \pmod{\ell}} \right\}, \]
is given by the complement of a countable union of closed subvarieties of $B$.
\end{corollary}

\begin{lemma}\label{lem:extendingRatEnd}
\begin{enumerate}
\item Suppose that $X\ra B$ is a map of complex varieties. Let $\etabar$ be the geometric generic point of $B$. If a very general fiber $X_t\in X$ admits a rational endomorphism of degree $\lambda$ then so does the geometric generic fiber $X_\etabar$.
\item Suppose that $K\subset L$ is an extension of algebraically closed fields. If $X_L$ admits a rational endomorphism of degree $\lambda$ then so does $X_K$.
\end{enumerate}
\end{lemma}

\begin{proof}
For part (1), by \cite[Lem. 2.1]{Vial13} there is a field isomorphism $\etabar\cong \CC$ such that the fiber product $X_\etabar \times_\etabar \CC$ is isomorphic to $X_t$ as varieties over $\CC$. Consider the rational endomorphism of schemes: $X_\etabar \cong X_t \dra X_t\cong X_\etabar.$ The first and last isomorphisms are isomorphisms of schemes, and the composition commutes with the map to $\etabar$. This gives the required endomorphism of $X_\etabar$.

For part (2), there is a finitely generated integral ring $A$ with $K\subset A \subset L$ which contains all the coefficients defining the rational endomorphism. Then the rational map defined over $L$ spreads out to a rational map $X_A\dra X_A$. Choosing a general $K$-point of $\Spec(A)$ gives the rational map $X_K\dra X_K$ which has the same degree.
\end{proof}

%%%%

\section{Resolving cyclic covers after base change}

The purpose of this section is to prove that $p$-cyclic covers admit sustained separably uniruled modifications in mixed characteristic $p$. To do this we need to resolve the typical singularities of these covers. Our resolutions involve certain weighted blowups, which at every step alternate between two types of isolated double point hypersurface singularities -- up to some mild quotient singularities (the equations of the singularities are given by $\eqref{singEq}$ in \S \ref{subsection:resolveCharP}). We believe that an alternative approach to these results is through resolution techniques coming from log geometry (see e.g. \cite{ILO14}). The arguments in this section are mostly computational in nature.

\subsection{Local equations for singularities of \texorpdfstring{$p$}{\texttwoinferior}-cyclic covers}

To start we recall the definition of a cyclic cover. Let $X$ be a scheme and $\Lc$ a line bundle on $X$. Let $s\in H^0(X,\Lc^{\otimes m})$ be a section. Let
\[
\LL=\Spec_{\Oc_X}\left(\oplus_{i\ge 0} \Lc^{-i}\cdot y^i\right)
\]
(resp. $\LL^{\otimes m}$) be the total space of the line bundle $\Lc$ (resp. $\Lc^{\otimes m}$). Then $s$ defines a section:
\begin{center}
\begin{tikzcd}
\LL^{\otimes m}\arrow[r]&X.\arrow[l,bend left=40,"s"]
\end{tikzcd}
\end{center}
There is also an $m$th power map: $\LL\xrightarrow{p_m}\LL^{\otimes m}$ which is a $\mu_m$-quotient.

\begin{definition}
The \textit{$m$-cyclic cover of $s$} is $Y:=p_m^{-1}(s(X))$. We say that the cyclic cover $Y$ has \textit{branch divisor} $(s=0)\subset X$.
\end{definition}

\noindent It follows that $Y\cong \Spec_{\Oc_X}\left(\oplus_{i\ge 0} \Lc^{-i} \cdot y^i/(y^m-s)\right)$.

Fix an odd prime $p$ and let $X_\zhens$ be an integral scheme that is smooth over $\zhens$. Let
\[
\sigma\cl Y_\zhens \ra X_\zhens
\]
be a $p$-cyclic cover with branch divisor $D=(s=0)$. Assume that
\begin{itemize}
\item the branch divisor $D$ is smooth over $\zhens$, and
\item the section $s\in H^0(X_\Fbar,\Oc_{X_\Fbar}(D))$ has \textit{nondegenerate critical points} \cite[\S20.3]{Kollar95}.
\end{itemize}
Let $q\in Y(\Fbar)$ and $\sigma(q)\in X(\Fbar)$ be points.

Now assume that
\[
p, x_1,\dots,x_n
\]
form a regular sequence for the regular local ring $S$ of $X$ at $\sigma(q)$. By the cyclic cover construction, $Y$ has local equation:
\[
0=y^p+u+f_1+f_2+f_3\subset S[y]
\]
where $y$ is the fiber coordinate, $u,f_1,f_2\in \zhens[x_1,\dots,x_n]$ are constant, linear, and quadratic polynomials in the $x_i$, and 
\[
f_3\in (x_1,\dots,x_n)^3 \subset S.
\]
Note that $Y_\Fbar$ is singular at $q \iff f_1=0\in S/p$. In addition, if $q\in Y_\Fbar$ is a singular point, then $u\in \zhens$ is a unit (as the branch divisor was assumed to be smooth).

If $q\in Y(\Fbar)$ is a singular point, then as $s$ has nondegenerate critical points, we have $f_2\in \Fbar[\xs]$ is a non-degenerate quadratic form. Thus, an appropriate ($\mathrm{GL}_n(\zhens)$-linear) change of the terms of the $x$-terms of the regular sequence of $S$ will diagonalize $f_2$ modulo $p$, so we may assume that $Y$ has an equation of the form
\[
y^p+u+x_1^2+\cdots+x_n^2+p(f_1+f_2)+f_3\in S[y],
\]
where $u\in \zhens$ is a unit, $\xs \in S$ is a regular sequence at $\sigma(q)$, $f_1$ and $f_2\in \zhens[\xs]$ are homogeneous linear and quadratic polynomials in the $\xs$ and $f_3$ is in the ideal $(\xs)^3\subset S$. The following lemma gives a local description of the singularities on $Y$:

\begin{lemma}
There are elements $\epsilon_1,\dots,\epsilon_n\in p\zhens$ such that if we set $\xbar_i := x_i-\epsilon_i$, then the regular sequence
\[
p,\xbar_1,\dots,\xbar_n\in S
\]
locally gives $Y$ the form
\[
0=y^p+u+\xbar_1^2+\dots+\xbar_n^2+pf_2+f_3\in S[y]
\]
where $u\in \zhens$ is a unit, $f_2\in \zhens[\xbar_1,\dots,\xbar_n]$ is quadratic in the $\xbar_i$s, and $f_3\in (\xbar_1,\dots,\xbar_n)^3\subset S$.
\end{lemma}

\begin{proof}
We define the elements $\epsilon_i$ by a convergent series:
\[
\epsilon_i=\sum\limits_{m=1}^\infty \epsilon_{i,m}\in\zhens
\]
for $\epsilon_{i,m}\in \zhens$. The idea is to iteratively complete the square for functions of the form:
\[
p^mf_1+x_1^2+\cdots+x_n^2.
\]
A technical problem is that $\zhens$ is not $p$-adically complete so one should be careful that the series converges inside $\zhens$. There are two solutions to this problem. The first is that every coefficient is actually defined over some ring that is a finite \'etale extension of the $p$-adics which is $p$-adically complete. The second, is to instead work with a $p$-adic completion of the strict Henselization of $\zhens$.

Consider an element of the form:
\[
F=u+x_1^2+\cdots+x_n^2+p^mf_1+pf_2+f_3\in S[y].
\]
Assume that $f_1=a_1x_1+\cdots+a_nx_n$, and set $\epsilon_{i,m}:=-p^{m} a_i/2$. Then the change of coordinates:
\[
\xbar_i+\epsilon_{i,m}=x_i
\]
gives
\begin{align*}
y^3+u&+(\xbar_1-p^ma_1/2)(\xbar_1+p^ma_1/2)+\cdots+(\xbar_n-p^ma_n/2)(\xbar_n+p^ma_n/2)+\\
&3f_2(\xbar_1-p^ma_1/2,\dots,\xbar_n-p^ma_n/2)+f_3(\xbar_1-p^ma_1/2,\dots,\xbar_n-p^ma_n/2)\in S[y].
\end{align*}
Expanding the terms $pf_2$ and $f_3$ into constant, linear, and quadratic terms in the $\xbar_i$ gives:
\[
y^3+u'+\xbar_1^2+\cdots+\xbar_n^2+p^{m+1}f'_1+ pf'_2+f'_3\in S[y].
\]
where $u'\in \zhens$ is a unit, $f'_1 (\text{resp. }f'_2) \in \zhens[\xbar_1,\dots,\xbar_n]$ is linear (resp. quadratic), and $f'_3$ is in the ideal $(\xbar_1,\dots,\xbar_n)^3\subset S.$ As $p^m|\epsilon_{i,m}$ we see that $\epsilon_i:=\sum \epsilon_{i,m}\in \zhens$ converges. Moreover, the p-adic norm of the linear term goes to 0.
\end{proof}

In order to prove Theorem~\ref{thm:SusRuledMod}, we need to show that after any extension of DVRs $\zhens\subset R'$ (with $\dim_\qhens(\Frac(R'))$ finite) there is a further extension $R' \subset R$ (with $\dim_\qhens(\Frac(R))$ finite) such that $Y_{R}$ admits ruled modifications. Let $\pi$ be a uniformizer of $R$. We know that there are finitely many singular points $q_i\in Y_\Fbar$, each of which (by Lemma 2.2) has local equation
\[
y^p+u_i+x_1^2+\cdots+x_n^2+pf_2+f_3=0
\]
where $u_i\in \zhens$ are units. We may assume that $R$ contains the $p$-th roots of all the $u_i$ and also assume that $k=\ord_\pi(p)$ is divisible by $2(p-1)$.

The regular local ring
\[
S_R := S\otimes_\zhens  R,
\]
comes with a regular sequence $\pi, x_1,\ldots,x_n$. Given a singularity of $Y_R$ with local equation above and unit term $u$, we may assume that $-\delta\in R$ is a $p$-th root of $u$. After changing the $y$-coordinate by $\delta$, we see that $Y_R$ is locally given by the equation
\[
0=y^p+py^2( \delta^2 y^{p-3}+\cdots+(p-1)\delta^{p-2}/2) + p y \delta^{p-1} + x_1^2+\cdots+x_n^2+pf_2+f_3\in S_R[y].
\]
We have now reduced to resolving singularities of this form.

\subsection{Charts for two types of weighted blowups.} To resolve the singularities of $Y_R$, we alternate between two types of weighted blowups. Here we will define and give affine charts for these weighted blowups. First we define these blowups for affine space.

The affine space $\AA^{n+1}_R$ has a regular sequence $\xs,y,\pi$ at the origin. For $\AA^{n+1}_R$ we consider the weighted blowups with weights: $((p-1)/2,\dots,(p-1)/2,1,1)$ and $((p+1)/2,\dots,(p+1)/2,1,1)$. We refer to these weighted blowups as
\[
\nu^-\cl\wb^-\ra \AA^{n+1}_R\text{, and }\nu^+\cl \wb^+\ra\AA^{n+1}_R.
\]
These charts will be very similar, so rather than describing them separately, we will just write
\[
\nu^\pm\cl \wb^\pm\ra \AA^{n+1}_R.
\]
To define $\wb^\pm$, we first consider the affine scheme:
\[
Z^\pm :=\Spec(R[x_1',\ldots,x_n',y',\omega,T]/(\pi=\omega T)).
\]
There is a map:
\[
Z^\pm\ra \AA^{n+1}_R
\]
induced by $(x_{1}, \ldots, x_{n}, y) \mapsto (x_1'T^{(p\pm 1)/2},\ldots, x_n'T^{(p\pm 1)/2},y'T)$. There is a $(\Gm)_R$-action on $Z^\pm$ defined by the following weights: $x_i'$ has weight $(p\pm 1)/2$, $y'$ has weight 1, $\omega$ has weight $1$, and $T$ has weight $-1$. The map $Z^\pm\ra \AA^{n+1}_R$ is $\Gm$-equivariant. We define the weighted blowup as
\[
\wb^\pm:=(Z^\pm\setminus (x_1'=\cdots=x_n'=y'=\omega=0))/\Gm,
\]
which comes with a natural projective morphism:
\[
\nu^\pm\cl \wb^\pm \ra \AA^{n+1}_R.
\]
For each coordinate $x_1',\dots,x_n',y',\omega$, there is an affine chart of $\wb^\pm$ defined by the $\Gm$-quotient of the complement of the associated divisor in $Z^\pm$.

The $x_1'\ne 0$ chart, denoted $\wb^\pm_{x_1}$ (similarly for the $x_i'$ charts) is computed by taking the spectrum of the ring of invariants:
\[
\wb_{x_1}^\pm :=\left(R[x_1',1/x_1',x_2',\dots,x_n',y',\omega,T]/(\pi=\omega T)\right)^\Gm.
\]
As $x_1'\ne 0$, if we let $z$ be a $(p\pm 1)/2$-th root of $x_1'$ this defines an \'etale cyclic cover:
\[
\Spec R[z,1/z,x_2',\dots,x_n',y',\omega,T]/(\pi=\omega T) \ra \wb^\pm_{x_1}.
\]
There is a $\Gm$-action on the new ring where $z$ has weight 1, and this \'etale cover descends to a finite (but no longer \'etale) degree $(p\pm 1)/2$ cyclic cover of the quotients:
\[
\left(R[x_1',1/x_1',x_2',\dots,x_n',y',\omega,T]/(\pi=\omega T)\right)^{\Gm} \subset \left(R[z,1/z,x_2',\dots,x_n',y',\omega,T]/(\pi=\omega T)\right)^{\Gm}.
\]
The ring of invariants of the finite cover can be computed to be
\[
R[\alpha_2,\dots,\alpha_n,\beta,\gamma,\zeta]/(\pi=\gamma \zeta)
\]
where
\[
\alpha_i=x_i'/z^{(p\pm 1)/2}, \beta = y'/z, \gamma=\omega/z,\text{ and }\zeta=z T.
\]

Define $\wbhat_{x_1}^\pm$ to be the spectrum of this ring of invariants. The ramification locus of
\[
\wbhat_{x_1}^\pm\ra \wb_{x_1}^\pm
\]
is the set of points in $\beta=\gamma=\zeta=0\subset \wbhat_{x_1}^\pm$. The map:
\[
\nuhat^\pm \cl \wbhat_{x_1}^\pm \ra \AA^{n+1}_R
\]
is given in coordinates by $(\zeta^{(p\pm 1)/2},\alpha_2 \zeta^{(p\pm1)/2},\dots,\alpha_n \zeta^{(p\pm1)/2},\beta\zeta).$ The preimage of the exceptional divisor of $\nu^\pm$ in $\wbhat_{x_1}^\pm$ is given by $\zeta=0$. The preimage of the strict transform of the central fiber is given by $\gamma=0$.

The $y'\ne 0$ chart, denoted $\wb_y^\pm$, is given by
\[
\wb_y^\pm:=\Spec\left(R[\alpha_1,\dots,\alpha_n,y,\gamma]/(\pi=y \gamma)\right)
\]
where $\alpha_i=x'_i/(y'^{(p\pm 1)/2})$, $y=y'T$, $\gamma=\omega/y'$. The map
\[
\nu_y^\pm \cl \wb_y^\pm \ra \AA^{n+1}_R
\]
is given in coordinates by $(\alpha_1 y^{(p \pm 1)/2},\dots,\alpha_n y^{(p \pm 1)/2},y)$. The exceptional divisor is $(y=0)$. The strict transform of the central fiber is $\gamma=0$.

The $\omega\ne 0$ chart, denoted $\wb_\omega^\pm$, is given by
\[
\wb_\omega^\pm:=\Spec\left(R[\alpha_1,\dots,\alpha_n,\beta]\right)
\]
The map
\[
\nu_\omega^\pm\cl \wb_\omega^\pm\ra \AA^{n+1}_R
\]
is given in coordinates by $(\alpha_1 \pi^{(p \pm 1)/2},\dots,\alpha_n \pi^{(p \pm 1)/2},\beta \pi)$. The exceptional divisor of $\nu_\omega^\pm$ is given by $\pi=0$ and the strict transform of the central fiber does not meet this chart.

Lastly, we need to define the weighted blowup of a smooth affine $R$-scheme. Assume $A$ is a finitely generated $R$-algebra and the map
\[
\Spec(A)\ra \Spec(R)
\]
is smooth. Let $q\in \Spec(A)$ be an $\Fbar$-point with maximal ideal defined by a regular sequence: $\mf=(\pi,x_1,\dots,x_n,y)$. Then (after possibly restricting to an open set) the map:
\[
\Spec(A)\ra \AA^{n+1}_R
\]
given in coordinates by $(x_1,\dots,x_n,y)$ is \'etale and there is only one point over $0$. The weighted blowup of $\Spec(A)$ with weights $((p\pm 1)/2,\dots,(p\pm 1)/2,1,1)$ at the maximal ideal $\mf$ with respect to the regular sequence $(\pi,x_1,\dots,x_n,y)$ is defined to be the fiber product:
\[
\begin{tikzcd}
\wb^\pm(A)\arrow[r]\arrow[d]& \wb^\pm\arrow[d] \\
\Spec(A)\arrow[r]& \AA^{n+1}_R.
\end{tikzcd}
\]
The map $\wb^\pm(A)\ra \wb^\pm$ is \'etale and the above charts give rise to charts for the weighted blowup of $\Spec(A)$ by base change.

\subsection{Resolving in mixed characteristic \texorpdfstring{$p$}{\texttwoinferior}}\label{subsection:resolveCharP}

The goal is to show that a hypersurface singularity of the form
\[
0=y^p+py^2( \delta^2 y^{p-3}+\cdots+(p-1)\delta^{p-2}/2) + p y \delta^{p-1} + x_1^2+\cdots+x_n^2+pf_2+f_3\in S_R[y]
\]
admits separably uniruled modifications. We will alternate between the types of weighted blowups defined in the previous section. Recall that $k=\ord_\pi(p)$, thus $p/\pi^i\in R$ for $i\le k$. Assume that $(2p-2)s<k$. We first define two series of equations:
\begin{equation}\label{singEq}
\begin{aligned}
F_s=&y^p+y^2g(y)p/\pi^{(2p-4)s}+ \delta^{p-1} y p /\pi^{(2p-2)s} + x_1^2+\cdots+x_n^2+pf_2+f_3.\\
G_s=&\pi y^p+y^2g(y) p/\pi^{(2p-4)s+p-3}+ \delta^{p-1} p y /\pi^{(2p-2)s+p-2} + x_1^2+\cdots+x_n^2+pf_2+f_3.
\end{aligned}
\end{equation}
We say that a scheme has a singularity with equation $F_s$ (resp. $G_s$) if (a) locally it can be embedded in a smooth affine $R$-scheme $\Spec(A)$; (b) there is an $\Fbar$ point in $\Spec(A)$ and a regular sequence $(\pi,x_1,\dots,x_n,y)$ such that the scheme is defined by $F_s=0$ (resp. $G_s=0$) where $\delta\in R$ is a unit, $g(y)\in R[y]$ is a polynomial in $y$, $f_2$ is a quadratic homogeneous polynomial in the $x_i$, and $f_3\in (x_1,\dots,x_n)^3\subset A$.

\begin{proposition}\label{prop:blowUpCases}
Assume that $(2p-2)s<k$.
\begin{enumerate}
\item The weighted blowup $\wb^-$ of a singularity with equation $F_s$ has 2 types of singularities: (a) a (non-isolated) cyclic quotient singularity with cyclic group of order $(p-1)/2$ and (b) an isolated singularity with equation $G_s$.
\item If $(2p-2)(s+1)<k$ then the weighted blowup $\wb^+$ of a singularity with equation $G_s$ has 2 types of singularities: (a) a (non-isolated) cyclic quotient singularity with cyclic group of order $(p+1)/2$ and (b) an isolated singularity with equation $F_{s+1}$.
\item If $(2p-2)(s+1)=k$ then the weighted blowup $\wb^+$ of a singularity with equation $G_s$ has only cyclic quotient singularities with cyclic group of order $(p+1)/2$.
\item All the exceptional divisors are ruled.
\end{enumerate}
\end{proposition}

\begin{proof}
We check this in our affine charts. The isolated singularity will show up in the affine chart $\wb^\pm_\omega$, the cyclic singularities will show up in the $\wb^\pm_{x_i}$-charts, and ruledness of exceptional divisors will be checked in the charts $\wb^\pm_y$.

We start by looking at the $\wb^\pm_\omega$-charts. By definition we are looking at a hypersurface in the affine scheme:
\[
\wb^\pm_\omega(A)=\Spec(A\otimes_{R[x_1,\dots,x_n,y]} R[\alpha_1,\dots,\alpha_n,\beta]).
\]
This is an \'etale cover of $\Spec(R[\alpha_1,\dots,\alpha_n,\beta])\cong \AA^{n+1}_R$. We compute the pullbacks of $F_s$.
\begin{align*}
(\nu^-_\omega)^*(F_s)=(\beta \pi)^p+&(\beta \pi)^2g(\pi\beta)p/\pi^{(2p-4)s} + \beta \pi \delta^{p-1}p/\pi^{(2p-2)s}  \\
&+ \pi^{p-1}(\alpha_1^2+\cdots+\alpha_n^2)+p\pi^{p-1}f_2+\pi^{3(p-1)/2}f'_3.
\end{align*}
where $f'_3\in (\alpha_1,\dots,\alpha_n)^3$. The strict transform of $F_s=0$ is given by dividing by $1/\pi^{p-1}$. This gives the equation:
\begin{align*}
(\nu^-)^*(F_s)/\pi^{p-1}=\pi \beta^p+&\beta^2g(\pi \beta)p/\pi^{(2p-4)s+p-3}+ \beta \delta^{p-1}4/\pi^{(2p-2)s+p-2}\\
& + \alpha_1^2+\cdots+\alpha_n^2+pf_2(\alpha_1,\dots,\alpha_n)+\pi^{(p-1)/2}f'_3.
\end{align*}
which is a singularity with equation $G_s$. Now we check that $(\pi,\alpha_1,\dots,\alpha_n,\beta)$ is the only singular point in this chart. We start by restricting to the exceptional divisor $(\pi=0)$. This gives the equation:
\[
\alpha_1^2+\cdots+\alpha_n^2=0\subset \AA^{n+1}_\Fbar.
\]
This is smooth away from points of the form $\mf=(\pi,\alpha_1,\dots,\alpha_n,\beta+\tau)$ (for some $\tau\in R$). Now we need to check regularity when $\tau$ is a unit in $R$. As $\tau$ is a unit:
\[
(\nu^-)^*(F_s)/\pi^{p-1} = \pi(-\tau)^p\ne0\in \mf/\mf^2
\]
so $(\nu^-)^*(F_s)/\pi^{p-1}$ is regular away from the closed point $(\pi,\alpha_1,\dots,\alpha_n,\beta)$.

A similar calculation (changing appropriate $-$ signs to $+$ signs) applies for part (2). In part (3), the strict transform has equation:
\begin{align*}
(\nu^+)^*(G_s)/\pi^{p+1}= \beta^p+&p\beta^2g(\pi\beta)/\pi^{(2p-4)k}+ p \beta \delta^{p-1}/\pi^{(2p-2)(s+1)}\\
& + \alpha_1^2+\cdots+\alpha_n^2+pf_2(\alpha_1,\dots,\alpha_n)+\pi^{(p+1)/2}f'_3.
\end{align*}
but now $\tau=p\delta^{p-1}/\pi^{(2p-2)(s+1)}$ is a unit in $R$. The exceptional divisor of the strict transform has equation:
\[
\beta^p+\tau\beta + \alpha_1^2+\dots+\alpha_n^2=0\subset \AA^{n+1}_\Fbar.
\]
The derivative $\partial/\partial\beta$ is never $0$ (as $\tau$ is a unit). Thus the exceptional divisor is a smooth Cartier divisor and the scheme is regular.

Now we examine the $\wb^\pm_y$-chart. This is defined as
\[
\wb^\pm_y(A)=\Spec\left(A\otimes_{R[\xs,y]} R[\alpha_1,\dots,\alpha_n,y,\gamma]/(\pi=y\gamma)\right),
\]
and there is an \'etale map $\wb^\pm_y(A)\ra \wb^\pm_y$. The strict transform of $F_s$ is
\begin{align*}
(\nu^-_y)^*(F_s)/y^{p-1}=y+& g(y)p/(y\gamma)^{(2p-4)s}y^{p-3}+ p \delta^{p-1}/(y\gamma)^{(2p-2)s}y^{p-2} \\
& + \alpha_1^2+\cdots+\alpha_n^2+pf_2(\alpha_1,\dots,\alpha_n)+y^{(p-1)/2}f'_3.
\end{align*}
Restricting to the exceptional divisor $y=0$ gives the equation:
\[
\alpha_1^2+\cdots+\alpha_n^2=0\subset\AA^{n+1}_\Fbar.
\]
So the only place we need to check regularity is at points with maximal ideal
\[
\mf=(\pi,\alpha_1,\dots,\alpha_n,y,\gamma+\tau)\subset A\otimes_{R[\xs,y]} R[\alpha_1,\dots,\alpha_n,y,\gamma].
\]
To check regularity, we need to show that $(\nu^-_y)^*(F_s)/y^{p-1}$ and $\pi- y \gamma$ are $\Fbar$-independent in $\mf/\mf^2$. This is straightforward as $(\nu^-_y)^*(F_s)/y^{p-1}\equiv y \pmod{\mf^2}$.

A similar calculation (changing appropriate $-$ signs to $+$ signs) applies for part (2). In part (3), the strict transform has equation:
\begin{align*}
(\nu^+_y)^*(F_s)/y^{p+1}=\gamma+&g(y)p/(y\gamma)^{(2p-2)s+p-2}y^p+p\delta^{p-1}/(y\gamma)^{(2p-2)s+p-2}y^p\\
&+\alpha_1^2+\cdots+\alpha_n^2+pf_2(\alpha_1,\dots,\alpha_n)+y^{(p+1)/2}f_3'.
\end{align*}
Observe that
\[
p\delta^{p-1}/y^{(2p-2)(s+1)}\gamma^{(2p-2)s+p-2}=\tau \gamma^p\in R[\alpha_1,\dots,\alpha_n,y,\gamma]/(y\gamma=\pi)
\]
for some unit $\tau \in R$. Restricting to the exceptional divisor gives:
\[
\gamma+\tau\gamma^p+\alpha_1^2+\cdots+\alpha_n^2=0\subset \AA^{n+1}_\Fbar.
\]
This is smooth, thus the total space is regular.

Now we want to check ruledness. In part (1) the exceptional divisor has equation:
\[
\alpha_1^2+\cdots+\alpha_n^2=0\subset \AA^{n+1}_\Fbar
\]
which is rational (as it is a quadric). In part (2) the exceptional divisor has equation:
\[
\gamma+\alpha_1^2+\cdots+\alpha_n^2=0\subset \AA^{n+1}_\Fbar
\]
which is also rational. Finally in part (3) the exceptional divisor has equation:
\[
\gamma+\tau\gamma^p+\alpha_1^2+\cdots+\alpha_n^2=0\subset \AA^{n+1}_\Fbar.
\]
This is birational to a hypersurface with homogeneous equation:
\[
(\gamma \theta^{p-1}+\tau\gamma^p+\theta^{p-2}(\alpha_1^2+\cdots+\alpha_n^2)=0)\subset \PP^{n+1}_\Fbar
\]
Note that this is a degree $p$ hypersurface with points of multiplicity $p-1$ (any point where $\gamma=\theta=\alpha_1^2+\cdots+\alpha_n^2=0$). Projecting from any such point gives a degree 1 map onto $\PP^{n+1}_\Fbar$, thus these exceptional divisors are rational as well.

It just remains to study the chart $\wb^\pm_{x_1}$. There is a (branched) cyclic cover
\[
\wbhat^\pm_{x_1}(A)\ra \wb^\pm_{x_i}(A).
\]
with degree $(p\pm 1)/2$. By definition we have
\[
\wbhat^\pm_{x_1}(A):=\Spec\left(A\otimes_{R[x_1,\dots,x_n,y]}R[\alpha_2,\dots,\alpha_n,\beta,\gamma,\zeta]/(\pi=\gamma \zeta)\right).
\]
The preimage of the strict transform of $F_s$ in $\wbhat^-_{x_1}(A)$ is
\begin{align*}
(\nuhat^-_{x_1})^*F_s/\zeta^{p-1}=\zeta \beta^{p-1}&+\beta^2 g(\zeta \beta) p/(\gamma\zeta)^{(2p-4)s}\zeta^{p-3}+\delta^{p-1}\beta p/(\gamma\zeta)^{(2p-2)s}\zeta^{p-2}\\
&1+\alpha_2^2+\cdots+\alpha_n^2+p f(1,\alpha_2,\dots,\alpha_n)+\zeta^{(p-1)/2}f'_3.
\end{align*}
We want to show this equation defines a regular scheme in $\wbhat^-_{x_1}(A)$. It suffices to show the intersection with the preimage of the exceptional divisor $(\zeta=0)$ is smooth. This restriction has equation:
\[
1+\alpha_2^2+\cdots+\alpha_n^2=0\subset \AA^{n+1}_\Fbar
\]
which is smooth. It follows that the strict transform of $F_s=0$ has cyclic quotient singularities in the chart $\wb^-_{x_1}$. Part (2) and (3) follow by a similar calculation.
\end{proof}

\begin{lemma}\label{lem:unirulingDegreeExc}
Suppose that $V$ and $W$ are normal $R$-schemes of finite type and there is a finite map $V\ra W$ of degree $d$. Suppose that $V$ has ruled modifications. Then for any birational map $W'\ra W$ from a normal $R$-scheme $W'$, every exceptional divisor has a uniruling of degree $\le d$.
\end{lemma}

\begin{proof}
Let $V'$ be the normalization of the closure of the graph of $V\dra W'$. There is some exceptional divisor of $V'\ra V$ (that is ruled) which dominates $E$ by a degree $\le d$ map.
\end{proof}

\begin{proof}[Proof of Theorem C in char. p>2]
We need to show that $Y_R$ admits separably uniruled modifications. By Proposition~\ref{prop:blowUpCases}, there is a birational model $Y'_R\ra R$ with ruled exceptional divisors such that $Y'_R$ only has cyclic quotient singularities by cyclic groups of order $\le (p+1)/2$. Then by Lemma~\ref{lem:unirulingDegreeExc} we are done, as any exceptional divisor over $Y_R$ is either exceptional over $Y'_R$ (which has separably uniruled modifications) or is exceptional for the map $Y'_R\ra Y_R$.
\end{proof}

\begin{proof}[Proof of Theorem C in char. 2]

In the central fiber at a singular point of the cyclic cover we have the equation:
\[
0=y^2+u+q+f_3\in (S/2)[y]
\]
where $S$ is a dimension $n+1$ regular local ring which is smooth over $\zhenstwo$ (and $2,\xs$ is a regular sequence), $u\in \Fbartwo$ is a unit, $q\in (S/2)[x_1,\dots,x_n]$ is a quadratic form with nondegenerate Hessian, and $f_3$ vanishes to order at least 3. As the characteristic is 2, $q$ having a nondegenerate Hessian implies that $n$ is even. By abuse of notation we lift this to an equation
\[
0=y^2+u+q+2f_1+f_3\in S[y]
\]
where $f_1$ (resp. $q$) is a linear (resp. quadratic) form in $\zhenstwo[\xs]$, $q/2$ has a nondegenerate Hessian, and $u\in \zhenstwo$ is a unit. Now base change along a finite extension $\zhenstwo\subset R$. After possibly extending so that $u$ has a square root $-\delta\in R$ we can make a change of coordinates to get the equation:
\[
0=y^2+(\delta y+f_1)2+ q +f_3\in S_R[y].
\]
where $\ord_\pi(2)=k$ for a uniformizer $\pi\in R$. Now consider ordinary blowup at the singular point $(\pi,\xs,y)$. The exceptional divisor of this blowup is given by
\[
y^2+q=0\subset \PP^{n+1}_\Fbartwo
\]
where the projective coordinates are $[\pi:x_1:\dots:x_n:y]$. This is only singular at the point $[1:0:\dots:0]$, and it follows that the blow up has a unique singularity of the form
\[
0=y^2+(\delta y+f_1)2/\pi+ q +f_3\in S_R[y]
\]
and the exceptional divisor is a quadric and thus ruled. Iterating the ordinary blowup brings us to the case that the exceptional divisor has the form:
\[
y^2+\delta y 2/\pi^{k-1}+f_1 2/\pi^{k-1} + q=0 \subset \PP^{n+1}_\Fbartwo.
\]
As $\delta y 2/\pi^{k-1}= \tau y\pi$ for some unit $\tau\in \Fbartwo$ the whole quadric has nondegenerate Hessian and is therefore smooth. So we see that iterating ordinary blowups of points resolves the singularities and the exceptional divisors are all quadrics in projective space and thus are ruled.
\end{proof}

\section{Endomorphisms in characteristic \texorpdfstring{$p$}{\texttwoinferior}}

The purpose of this section is to prove the nonexistence of separable rational endomorphisms of degree at least 2 for $p$-cyclic covers in characteristic $p$. We start by defining pull backs for rational endomorphisms.

\begin{definition}
Let $\Phi\cl Y\dra Y$ be a rational endomorphism of a proper, normal, $\QQ$-factorial variety over an algebraically closed field. Let $\Gamma_\Phi$ be the normalization of the closure of $\Phi$ in $Y\times Y$ with projections $p_1$ and $p_2$ of degree 1 and $\deg(\Phi)$ respectively. We define the pullback along $\Phi$:
\[
\Phi^*\cl \Pic(Y)\otimes \QQ \ra \Pic(Y)\otimes \QQ
\]
to be the composition:
\[
\Pic(Y)\otimes\QQ \xrightarrow{p_2^*}\Pic(\Gamma_\Phi)\otimes \QQ \hookrightarrow \Cl(\Gamma_\Phi)\otimes \QQ\xrightarrow{p_{1*}}\Cl(Y)\otimes \QQ \cong \Pic(Y)\otimes \QQ.
\]
\end{definition}

We would like to thank Nguyen-Bac Dang and John Lesieutre for suggesting the use of the Khovanskii-Teissier inequalities to prove the following proposition.

\begin{proposition}\label{prop:multPic}
Let
\[
\phi \colon Y \dashrightarrow Y
\]
be a rational endomorphism of degree $\lambda\ge 1$ of a normal $\QQ$-factorial projective variety $Y$ of dimension $n$ over an algebraically closed field of arbitrary characteristic. If $\Pic(Y)\otimes \QQ \cong \QQ$ then
\[
\phi^{\ast}\cl \Pic(Y)\otimes \QQ \ra \Pic(Y)\otimes \QQ
\]
is multiplication by $a$ where $a \geq \sqrt[n]{\lambda}$.
\end{proposition}

\begin{proof}
Recall that for any 2 nef divisor classes $H_1$ and $H_2$ on a variety $Y$ we have the inequalities
\[
(H_1^i\cdot H_2^{n-i})^{n} \ge (H_1^n)^i \cdot (H_2^n)^{n-i}
\]
(see Corollary 1.6.3(ii) and Remark 1.6.5 of \cite{Lazarsfeld04} -- this is closely related to the log concavity of dynamical degrees). Let $\Gamma\ra Y\times Y$ be the graph of $\phi$ with projections $p_1\cl \Gamma\ra Y$ and $p_2\cl \Gamma\ra Y$ of degrees $1$ and $\lambda$ respectively. Let $H$ be an ample generator for $\Pic(Y)$. Set $H_i=p_i^*H$. Applying the above inequality for $i=n-1$ gives
\[
(H_1^{n-1}\cdot H_2)^n\ge (H_1^n)^{n-1} \cdot (H_2^n).
\]
Applying the projection formula to the left hand side gives
\[
(a H^n)^n\ge (H^n)^{n-1}\cdot \lambda (H^n),
\]
which reduces to $a\ge \sqrt[n]{\lambda}.$
\end{proof}

Now we recall Koll\'ar's result on the existence of holomorphic forms on cyclic covers in characteristic $p$. Let $X\subset \PP^{n+1}_\Fbar$ be a degree $e$ hypersurface and let $D=(s=0)\subset X$ be the divisor for some $s\in H^0(X,\Oc_X(pe))$. Assume that $s$ has nondegenerate critical points (\cite[17.3]{Kollar96}). Then the cyclic cover of $X$ branched at $D$ admits a resolution
\begin{equation}
\begin{tikzcd}
Z\arrow[r,"\mu"]\arrow[rr,swap,bend right=40,"\nu"] & Y\arrow[r, "\sigma"] &X,
\end{tikzcd}
\end{equation}
obtained by blowing up finitely many isolated singular points. When $n\ge 3$ there is an injection:
\[
\Mc:=\nu^*\left(\Oc_X(pe+e-n-2)\right)\hookrightarrow \Omega_Z^{n-1}.
\]

\begin{remark}\label{rem:degNumerics}
If we write $d=pe+f$ where $0 \leq f \leq p-1$, then $d\ge p\lceil (n+3)/(p+1)\rceil$ if and only if $pe+e-n-2>0$. Indeed, we have
\[ e+\frac{f}{p} \ge \left\lceil \frac{n+3}{p+1} \right\rceil \iff e\ge \left\lceil \frac{n+3}{p+1} \right\rceil \iff e\ge \frac{n+3}{p+1} \iff pe+e-n-2>0.\]
\end{remark}

\begin{lemma}\label{lem:YisQfactorial}
For $n\ge 3$ we have $\Pic(Y)\otimes \QQ \cong \Cl(Y)\otimes \QQ \cong \QQ$; in particular $Y$ is $\QQ$-factorial.
\end{lemma}

\begin{proof}
By \cite[IV.3]{Hartshorne70} the hypersurface $X\subset \PP_\Fbar^{n+1}$ satisfies $\Pic(X)\cong \ZZ$. The cyclic cover construction gives
\[
Y\subset \mathbb{O}(e)
\]
(the total space of $\Oc_X(e)$), and the composition
\[
\Pic(Y)\xrightarrow{\sigma_*}\Pic(X)\cong \ZZ\xrightarrow{\sigma^*} \Pic(Y)
\]
is multiplication by $p$. It follows that $\Pic(Y)\otimes \QQ\cong \Pic(X)\otimes \QQ\cong \QQ$.
\end{proof}

\begin{proposition}\label{prop:NoSep}
If $pe+e-n-2>0$ (i.e. if $\Mc$ is big and nef on $Z$) then any separable rational endomorphism $\phi\cl Y\dra Y$ has degree $0$ or $1$.
\end{proposition}

\begin{proof}
The idea is that if $\deg(\phi)\ge 2$ then pulling back $\Mc$ under a large iterate $\phi^{\circ k}$ gives an arbitrarily positive line bundle inside $\mybigwedge^{n-1}\Omega_Z$, which is impossible.

Now let $H_Y$ be an ample line bundle on $Y$ and let $H_Z:=\mu^*H_Y$. We consider the following invariant:
\[
S := \sup \left\{ (\Lc \cdot H_{Z}^{n-1}) \mid \Lc \text{ is a line bundle on } Z \text{ and } \Lc \hookrightarrow \mybigwedge^{n-1}\Omega_{Z} \right\} \in \ZZ.
\]
As $\Mc\hookrightarrow \mybigwedge^{n-1}\Omega_Z$ we have $S>0$. We also claim that $S$ is bounded. Indeed, for any line bundle $\Lc \hookrightarrow \mybigwedge^{n-1}\Omega_{Z}$ we have
\[
(\mu_*\Lc)^{\vee\vee} \hookrightarrow (\mu_*\mybigwedge^{n-1}\Omega_{Z})^{\vee\vee}.
\]
The sheaf $(\mu_*\Lc)^{\vee\vee}$ is a rank one reflexive sheaf on $Y$. By ampleness there exists some $N_1>0$, $N_2>0$ such that
\[
(\mu_*\mybigwedge^{n-1}\Omega_{Z})^{\vee\vee}\hookrightarrow ((H_Y)^{\otimes N_1})^{\oplus N_2}.
\]
It follows that $(\mu_*\Lc)^{\vee\vee}$ injects into one of the factors $H_Y^{\otimes N_1}$ which bounds the degree $(\mu_*\Lc)^{\vee\vee}\cdot H_Y^{n-1}$. Lastly, applying projection formula gives:
\[
(\mu_*\Lc)^{\vee\vee}\cdot H_Y^{n-1}=\Lc\cdot H_Z^{n-1}.
\]
Therefore, $S$ is bounded.

Now assume for contradiction that there is a separable rational endomorphism $\phi$ of $Y$ with $\deg(\phi)=\lambda\ge 2$. Let $\Gamma_Y$ (resp. $\Gamma_Z$) be the normalization of the graph of $\phi^{\circ k} \colon Y \dra Y$. Then we have a diagram:
\begin{center}
\begin{tikzcd}[row sep=small]
& \Gamma_{Z} \arrow[ld, swap, "\psi_{1}"] \arrow[rd, "\psi_{2}"] \arrow[dd, "\gamma"] & \\
Z \arrow[dd, swap, "\mu"] & & Z \arrow[dd, "\mu"] \\
& \Gamma_{Y} \arrow[ld, "\pi_{1}"] \arrow[rd, swap, "\pi_{2}"] & \\
Y\arrow[rr,dashed,swap,"\phi^{\circ k}"] & & Y.
\end{tikzcd}
\end{center}
Let $\phi_Z^{\circ k}$ denote the rational endomorphism of $Z$. This gives rise to a diagram of the associated Picard groups (and divisor class groups).
\begin{equation}\label{picCommutes}
\begin{tikzcd}
\Pic(Y)\otimes \QQ \arrow[r,"\pi_2^*"]\arrow[d,"\mu^*"] & \Pic(\Gamma_Y)\otimes\QQ \arrow[r,hook]
\arrow[d,"\gamma^*"]& \Cl(\Gamma_Y)\otimes \QQ\arrow[r,"\pi_{1*}"] & \Pic(Y)\otimes \QQ\\
\Pic(Z)\otimes \QQ\arrow[r,"\psi_2^*"] & \Pic(\Gamma_Z)\otimes \QQ\arrow[r,hook] & \Cl(\Gamma_Z)\otimes \QQ\arrow[r,"\psi_{1*}"]\arrow[u,"\gamma_*"] & \Pic(Z)\otimes \QQ\arrow[u,"\mu_*"]
\end{tikzcd}
\end{equation}
This diagram is commutative (the left and right squares commute by functoriality and the central square commutes as $\gamma$ is birational). The composition of the top row is $(\phi^{\circ k})^{\ast}$ and the composition of the bottom row is $(\phi^{\circ k}_{Z})^{\ast}$.

Let $\Phi$ be any positive degree separable rational endomorphism $\Phi\cl Z\dra Z$ of $Z$ and let $\Gamma_\Phi$ be the normalization of the graph of $\Phi$ in $Z\times Z$. For any line bundle $\Lc \subset \mybigwedge^{n-1}\Omega_Z$ there are injections:
\[
p_2^*\Lc \hookrightarrow p_2^*\mybigwedge^{n-1}\Omega_Z \hookrightarrow \mybigwedge^{n-1}\Omega_{\Gamma_\Phi}.
\]
and the composition is injective. Pushing forward this uniquely defines an injection:
\[
\left(p_{1*}(p_2^*\Lc)\right)^{\vee\vee} \hookrightarrow \left(p_{2*}\left(\mybigwedge^{n-1}\Omega_{\Gamma_\Phi}\right) \right)^{\vee\vee}\cong \mybigwedge^{n-1}\Omega_Z.
\]
(As both sheaves are $S_2$ the last isomorphism can be checked away from the codimension $\ge 2$ locus where $p_1^{-1}$ is undefined.) We also have that $\left(p_{1*}(p_2^*\Lc)\right)^{\vee\vee}= \Phi^*(\Lc)\in \Pic(W)\otimes \QQ$.

Thus for each $k>0$ there is an injection
\[
(\phi^{\circ k}_{Z})^{\ast}\Mc \hookrightarrow \mybigwedge^{n-1} \Omega_{Z},
\]
and we know $\Mc=\mu^*\Nc$ for some ample line bundle $\Nc$ on $Y$. By commutativity of the outer arrows of diagram~\ref{picCommutes},
\begin{align*}
S\ge (\phi^{\circ k}_{Z})^{\ast}\Mc \cdot H_{Z}^{n-1} &= \mu_{\ast}
\left ((\phi^{\circ k}_{Z})^{\ast} \mu^{\ast}\Nc \right)\cdot H_{Y}^{n-1} \\
&= (\phi^{\circ k}_{Y})^{\ast}\Nc \cdot H_{Y}^{n-1} \\
&= a\Nc \cdot H_{Y}^{n-1}
\end{align*}
for some constant $a \geq \sqrt[n]{\lambda^{k}}$ (by Proposition~\ref{prop:multPic} and Lemma~\ref{lem:YisQfactorial}). For $k \gg 0$, this is arbitrarily large, which is a contradiction.
\end{proof}

\section{Degrees of endomorphisms of complex hypersurfaces}

Let $X_\zhens\subset \PP^{n+1}_\zhens$ be a degree $e$ hypersurface that is smooth over $\zhens$ with $n\ge 3$. Let
\[
s\in H^0(X_\zhens,\Oc_{X_\zhens}(pe))
\]
be a section such that $D_\zhens=(s=0)\subset X_\zhens$ is smooth over $\zhens$ and such that 
\[
s|_{X_\Fbar}\in H^0(X_\Fbar,\Oc_{X_\Fbar}(pe))
\]
has nondegenerate critical points.

\begin{proposition}\label{prop:cyclicCoverRatEnd}
Let $Y_\CC$ be the $p$-cyclic cover of $X_\CC$ branched at $D_\CC$. If
\[
pe\ge p \lceil (n+3)/(p+1)\rceil
\]
and $\phi\cl Y_\CC\dra Y_\CC$ is any rational endomorphism of 
$Y_\CC$ of degree $\lambda$ then
\[
\lambda\equiv 0\text{ or }1 \pmod{p},
\]
and every uniruling degree of $Y_\CC$ is divisible by $p$.
\end{proposition}

\begin{proof}
By Lemma~\ref{lem:extendingRatEnd}, as $X_\CC$ and $Y_\CC$ are defined over $\qhens$, there is an endomorphism $\phi$ of $Y_\CC$ of degree $\lambda$ which is defined over a finite extension of $\qhens$. Theorem~\ref{thm:SusRuledMod} implies that $Y_\zhens$ has sustained separably uniruled modifications. By Proposition~\ref{prop:NoSep} any rational endomorphism of $Y_\Fbar$ of degree greater than 1 is inseparable. Thus the degree of any rational endomorphism of $Y_\Fbar$ is 0 or 1$\pmod{p}$. By \cite[Lem. 7]{Kollar95}, $Y_\Fbar$ is not separably uniruled. Thus by Theorem~\ref{SpecializeRatEnd} we have $\lambda\equiv 0\text{ or }1 \pmod{p}$.
\end{proof}

\begin{proof}[Proof of Theorem~\ref{mainThm}]
To start we prove the theorem when $d=pe$. By Mori's construction \cite{Mori75} (or see \cite[\S 5]{Kollar96}) there is a smooth family of complex varieties which contains all smooth hypersurfaces of degree $pe$ in $\PP^{n+1}_\CC$ and all smooth degree $p$ cyclic covers of smooth degree $e$ hypersurfaces in $\PP^{n+1}_\CC$. If $d\ge p\lceil (n+3)/(p+1)\rceil$ then by Proposition~\ref{prop:cyclicCoverRatEnd} there are $p$-cyclic covers of degree $e$ hypersurfaces such that every rational endomorphism has degree $\equiv$ 0 or 1$\pmod{p}$ and every uniruling degree is divisible by $p$. Thus, if $X$ is a very general hypersurface of degree $pe$, then every rational endomorphism has degree $\equiv 0$ or $1\pmod{p}$.

Now we consider the general case with $d=pe+f$. Consider a pencil of hypersurfaces of degree $d$ such that the general member is smooth and the central fiber is a union of hyperplanes and a very general hypersurface of degree $pe$ meeting with simple normal crossings. By the previous paragraph and Remark~\ref{rem:degNumerics} every rational endomorphism of the degree $pe$ hypersurface has degree $\equiv$ 0 or 1$\pmod{p}$. By \cite[Lem. V.5.14.5]{Kollar96} every uniruling degree of such a degree $pe$ hypersurface is divisible by $p$. Therefore, by Theorem~\ref{SpecializeRatEnd}(2) and Lemma~\ref{lem:extendingRatEnd} every rational endomorphism of a very general hypersurface of degree $d$ has degree $\equiv 0$ or $1\pmod{p}$.
\end{proof}

%%%%%%%%%%%%%%%%
%
%  REFERENCES
%
%%%%%%%%%%%%%%%%

\bibliographystyle{siam} 
\bibliography{Biblio}

\footnotesize{
\textsc{Department of Mathematics, Stony Brook University, Stony Brook, New York 11794} \\
\indent \textit{E-mail address:} \href{mailto:nathan.chen@stonybrook.edu}{nathan.chen@stonybrook.edu}

\textsc{Department of Mathematics, University of California San Diego, La Jolla, California 92093} \\
\indent \textit{E-mail address:} \href{mailto:dstapleton@ucsd.edu}{dstapleton@ucsd.edu}
}

\end{document}